\documentclass[twoside,11pt,a4paper]{article}
\usepackage{fullpage}
\usepackage{graphicx}
\usepackage{amsmath,amsthm}
\usepackage{amssymb,amsfonts,amsbsy}
\usepackage{dsfont,eufrak,mathrsfs}
\usepackage{calligra}
\usepackage{color}
\usepackage[T1]{fontenc}

\DeclareMathAlphabet{\mathpzc}{OT1}{pzc}{m}{it}
\DeclareFontShape{T1}{calligra}{m}{n}{<->s*[1.2]callig15}{}
\DeclareMathAlphabet{\mathcalligra}{T1}{calligra}{m}{n}

\newtheorem{Def}{Definition}[section]
\newtheorem*{Teono}{Theorem}

\newtheorem{Lema}[Def]{Lemma}

\newtheorem{Cor}[Def]{Corollary}

\newtheorem{Rem}[Def]{Remark}

\newtheorem{Theo}{Theorem}

\newtheorem{Coro}[Theo]{Corollary}
\newtheorem{Quest}[Theo]{Question}

\newcommand*{\fun}[3]{#1: #2 \rightarrow #3}
\newcommand*{\ccomp}[1]{\mathcal{C}(#1)}

\newcommand*{\htcomp}[1]{\mathcal{HT}(#1)}
\newcommand*{\gcomp}[1]{\mathcal{G}(#1)}
\newcommand*{\htstrat}[2]{\mathcal{HT}_{#1}(#2)}

\newcommand*{\EMod}[1]{\mathrm{Mod}^{*}(#1)}

\newcommand*{\lk}[1]{\mathcalligra{lk} \hspace{0.12cm}(#1)}

\newcommand*{\col}[1]{\thicksim_{#1}}
\newcommand*{\ncol}[1]{\nsim_{#1}}
\newcommand*{\piHT}{\pi_{\mathcal{HT}}}
\newcommand*{\piC}{\pi_{\mathcal{C}}}

\newcommand*{\ColonEqq}{\mathrel{\mathop:}=}
\newcommand*{\biprime}{\prime\prime}
\newcommand*{\veps}{\varepsilon}

\newcommand*{\epa}{edge-preserving alternating }
\newcommand*{\Aut}[1]{\mathrm{Aut}(#1)}

\author{Jes\'{u}s Hern\'{a}ndez Hern\'{a}ndez}
\title{Alternating maps on Hatcher-Thurston graphs}
\date{}

\begin{document}
\maketitle
\begin{abstract}
 Let $S_{1}$ and $S_{2}$ be connected orientable surfaces of genus $g_{1}, g_{2} \geq 3$, $n_{1},n_{2} \geq 0$ punctures, and empty boundary. Let also $\fun{\varphi}{\htcomp{S_{1}}}{\htcomp{S_{2}}}$ be an edge-preserving alternating map between their Hatcher-Thurston graphs. We prove that $g_{1} \leq g_{2}$ and that there is also a multicurve of cardinality $g_{2} - g_{1}$ contained in every element of the image. We also prove that if $n_{1} = 0$ and $g_{1} = g_{2}$, then the map $\widetilde{\varphi}$ obtained by filling the punctures of $S_{2}$, is induced by a homeomorphism of $S_{1}$.
\end{abstract}
\section*{Introduction}
\indent Suppose $S_{g,n}$ is an orientable surface of finite topological type, with genus $g \geq 3$, empty boundary, and $n \geq 0$ punctures. The (extended) \textit{mapping class group} is the group of isotopy classes of self-homeomorphisms of $S_{g,n}$.\\
\indent In 1980 (see \cite{HTComplex}), Hatcher and Thurston introduce the Hatcher-Thurston complex of a surface, which is the $2$-dimensional CW-complex whose vertices are multicurves called \textit{cut systems}, $1$-cells are defined as elementary moves between cut systems, and $2$-cells are defined as appropriate ``triangles'', ``squares'' and ``pentagons''. See Section \ref{prelim} for the details. They used this complex to prove that the index 2 subgroup of $\EMod{S_{g,n}}$ of orientation preserving isotopy classes, is finitely presented. The $1$-skeleton of this complex is called the \textit{Hatcher-Thurston graph}, which we denote by $\htcomp{S_{g,n}}$.\\
\indent There is a natural action of $\EMod{S_{g,n}}$ on the Hatcher-Thurston complex by automorphisms, and in \cite{IrmakKorkmaz} Irmak and Korkmaz proved that the automorphism group of the Hatcher-Thurston complex is isomorphic to $\EMod{S_{g,n}}$. Inspired by the different results in combinatorial rigidity on other simplicial graphs (like the curve graph in \cite{Shack} and \cite{JHH2}, and the pants graph in \cite{AraPants}), we obtain analogous results concerning simplicial maps between Hatcher-Thurston graphs.\\
\indent Let $S_{1} = S_{g_{1},n_{1}}$ and $S_{2} = S_{g_{2},n_{2}}$ with $g_{1}, g_{2} \geq 2$ and $n_{1}, n_{2} \geq 0$. A simplicial map $\fun{\varphi}{\htcomp{S_{1}}}{\htcomp{S_{2}}}$ is \textit{alternating} if the restriction to the star of any vertex, maps cut systems that differ in exactly $2$ curves to cut systems that differ in exactly $2$ curves. See Section \ref{prelim} for the details. In Section \ref{chap4sec2} we prove our first result concerning this type of map:
\begin{Theo}\label{TheoA}
 Let $S_{1}$ and $S_{2}$ be connected orientable surfaces, with genus $g_{1}, g_{2} \geq 2$ respectively, with empty boundary and $n_{1}, n_{2} \geq 0$ punctures respectively. Let $\fun{\phi}{\htcomp{S_{1}}}{\htcomp{S_{2}}}$ be an edge-preserving and alternating map. Then we have the following:
 \begin{enumerate}
  \item $g_{1} \leq g_{2}$.
  \item There exists a unique multicurve $M$ in $S_{2}$ with $g_{2} - g_{1}$ elements such that $M \subset \phi(C)$ for all cut systems $C$ in $S_{1}$.
 \end{enumerate}
\end{Theo}
\indent A consequence of this theorem is that whenever we have an \epa map $\fun{\phi}{\htcomp{S_{1}}}{\htcomp{S_{2}}}$ (where the conditions of Theorem \ref{TheoA} are satisfied), we can then induce an \epa map $\fun{\varphi}{\htcomp{S_{1}}}{\htcomp{S_{2} \backslash M}}$ where $M$ is the multicurve obtained by Theorem \ref{TheoA}, and $S_{2} \backslash M$ is connected (due to the nature of Theorem \ref{TheoA}) and has genus $g_{1}$. This means we can focus solely on the case where $g_{1} = g_{2}$. However, due to the nature of $\htcomp{S_{1}}$ and the techniques available right now, it is quite difficult to study these maps if $n_{1} > 0$, and it is possible to have \epa maps if $n_{1} < n_{2}$ that are obviously not induced by homeomorphisms, e.g. creating $n_{2} - n_{1}$ punctures in $S_{1}$.\\
\indent A way around this particular complication is wondering if this is the only way for the edge-preserving alternating maps to be not induced by homeomorphisms, leading to the following question:
\begin{Quest}\label{QHTS1S2}
  Let $S_{1}$, $S_{2}$ and $S_{3}$ be connected orientable surfaces, with genus $g \geq 3$, $n_{1}, n_{2} \geq 0$ punctures for $S_{1}$ and $S_{2}$ respectively, and assume $S_{3}$ is closed. Let $\fun{\phi}{\htcomp{S_{1}}}{\htcomp{S_{2}}}$ be an edge-preserving alternating map. Is there a way to induce a well-defined map $\fun{\varphi}{\htcomp{S_{3}}}{\htcomp{S_{3}}}$ from $\phi$ by filling the punctures of $S_{1}$ and $S_{2}$? If so, is $\varphi$ induced by a homeomorphism?
\end{Quest}
\indent In Section \ref{chap4sec3} we answer this question for a particular case. If $\piHT$ is the map induced by filling the punctures of $S_{2}$, we have the following result:
\begin{Theo}\label{TheoB}
 Let $S_{1}$ and $S_{2}$ be connected orientable surfaces, with genus $g \geq 3$ and empty boundary, and assume $S_{1}$ is closed. Let $\fun{\phi}{\htcomp{S_{1}}}{\htcomp{S_{2}}}$ be an edge-preserving and alternating map. Then $$\fun{\widetilde{\phi}\ColonEqq \piHT \circ \phi}{\htcomp{S_{1}}}{\htcomp{S_{1}}}$$ is induced by a homeomorphism of $S_{1}$.
\end{Theo}
\indent This implies that the only way to obtain a map from $\htcomp{S_{1}}$ to $\htcomp{S_{2}}$ that is edge-preserving and alternating, is to use a homeomorphism of $S_{1}$ and then puncture the surface to obtain $S_{2}$.\\
\indent Theorem \ref{TheoB} is proved by using $\phi$ to induce maps between the underlying curves of the cut systems, and eventually induce an edge-preserving self-map of the curve graph of $S_{1}$ (see Section \ref{chap4sec3} for the details). Then, by the Theorem A of \cite{JHH2} (the second article of a series of which this work is also a part) we have that said self-map is induced by a homeomorphism.\\
\indent Later on, in Section \ref{chap4sec4} we prove a consequence of Theorems \ref{TheoA} and \ref{TheoB} concerning isomorphisms and automorphisms between Hatcher-Thurston graphs.
\begin{Coro}\label{CoroC}
 Let $S_{1}$ and $S_{2}$ be connected orientable surfaces, with genus $g_{1}, g_{2} \geq 2$ respectively, with empty boundary and $n_{1}, n_{2} \geq 0$ punctures respectively. If $\fun{\phi}{\htcomp{S_{1}}}{\htcomp{S_{2}}}$ is an isomorphism, we have that $\phi$ is an alternating map and $g_{1} = g_{2}$. Moreover, this implies that if $S = S_{g,0}$ with $g \geq 3$, then $\Aut{\htcomp{S}}$ is isomorphic to $\EMod{S}$.
\end{Coro}
\indent We must remark that this work is the published version of the fourth chapter of the author's Ph.D. thesis (see \cite{Thesis}), and the results here presented are dependent on the results found in \cite{JHH2}, which is the published version of the third chapter. There we prove that for any edge-preserving map between the curve graphs of a priori different surfaces (with certain conditions on the complexity and genus for the surfaces) to exist, it is necessary that the surfaces be homeomorphic and that the edge-preserving map be induced by a homeomorphism between the surfaces.\\[0.3cm]
\textbf{Acknowledgements:} The author thanks his Ph.D. advisors, Javier Aramayona and Hamish Short, for their very helpful suggestions, talks, corrections, and specially for their patience while giving shape to this work.
\section{Preliminaries and properties}\label{prelim}
\indent In this section we give several definitions and prove several properties of the Hatcher-Thurston graph. Here we suppose $S = S_{g,n}$ with genus $g \geq 1$ and $n \geq 0$ punctures.\\
\indent A \textit{curve} $\alpha$ is a topological embedding of the unit circle into the surface. We often abuse notation and call ``curve'' the embedding, its image on $S$ or its isotopy class. The context makes clear which use we mean.\\
\indent A curve is \textit{essential} if it is neither null-homotopic nor homotopic to the boundary curve of a neighbourhood of a puncture.\\
\indent The (geometric) intersection number of two (isotopy classes of) curves $\alpha$ and $\beta$ is defined as follows: $$i(\alpha,\beta) \ColonEqq \min \{|a \cap b| : a \in \alpha, b \in \beta\}.$$
\indent Let $\alpha$ and $\beta$ be two curves on $S$. Here we use the convention that $\alpha$ and $\beta$ are \textit{disjoint} if $i(\alpha, \beta) = 0$ \textbf{and} $\alpha \neq \beta$.\\
\indent A \textit{multicurve} $M$ is either a single curve or a set of pairwise disjoint curves. A \textit{cut system} $C$ of $S$ is a multicurve of cardinality $g$ such that $S \backslash C$ is connected.\\
\indent Similarly, a curve $\alpha$ is separating if $S\backslash \{\alpha\}$ is disconnected, and is nonseparating otherwise. Note that a cut system can only contain nonseparating curves, and also $S \backslash C$ has genus zero, thus a cut system $C$ can be \textbf{characterized} as a maximal multicurve such that $S \backslash C$ is connected.\\
\indent Two cut systems $C_{1}$ and $C_{2}$ are related by an elementary move if they have $g-1$ elements in common and the remaining two curves intersect once.\\
\indent The \textit{Hatcher-Thurston graph} $\htcomp{S}$ is the simplicial graph whose vertices correspond to cut systems of $S$, and where two vertices span an edge if they are related by an elementary move. We will denote by $\mathcal{V}(\htcomp{S})$ the set of vertices of $\htcomp{S}$.\\
\indent If $M$ is a multicurve on $S$, we will denote by $\htstrat{M}{S}$ the (possibly empty) full subgraph of $\htcomp{S}$ spanned by all cut systems that contain $M$.
\begin{Rem}\label{FirstRem}
 Let $M$ and $M^{\prime}$ be multicurves on $S$ such that neither $\htstrat{M}{S}$ nor $\htstrat{M^{\prime}}{S}$ are empty graphs. Then $\htstrat{M}{S} \subset \htstrat{M^{\prime}}{S}$ if and only if $M^{\prime} \subset M$. Also, if $M$ is a multicurve such that $\htstrat{M}{S}$ is nonempty, then $\htstrat{M}{S}$ is naturally isomorphic to $\htcomp{S \backslash M}$.
\end{Rem}
Recalling previous work on the Hatcher-Thurston complex we have the following lemma.
\begin{Lema}[\cite{Wajnryb}]
 Let $S$ be an orientable connected surface of genus $g \geq 1$, with empty boundary and $n \geq 0$ punctures. Then $\htcomp{S}$ is connected.
\end{Lema}
\indent Note that this lemma and Remark \ref{FirstRem} imply that if $M$ is a multicurve on $S$ such that $S \backslash M$ is connected, then $\htstrat{M}{S}$ is connected.
\subsection{Properties of $\htcomp{S}$}\label{chap4sec1subsec1}
\indent Let $C$ be a cut system on $S$, and denote by $\lk{C}$ the full subgraph spanned by the set of cut systems on $S$ that are adjacent to $C$ in $\htcomp{S}$ (often called the \textit{link of} $C$ \textit{in} $\htcomp{S}$). Intuitively, we want to relate the elements of $\lk{C}$ that are obtained by replacing the same curve of $C$; this is done defining the relation $\col{C}$ in $\lk{C}$ by $$C_{1} \col{C} C_{2} \Leftrightarrow C_{1} \cap C = C_{2} \cap C.$$
We can easily check $\col{C}$ is an equivalence relation, and two cut systems are related in $\lk{C}$ if they are obtained by replacing the same curve of $C$ as was desired. The equivalence classes of this relation will be called \textit{colours}.\\
\indent This definition implies that in $\lk{C}$ there are $g$ colours, each corresponding to a curve in $C$ that was substituted; thus, we use the elements of $C$ to index these colours.
\begin{Rem}\label{g-2curves}
 We should note that if $C_{1}, C_{2} \in \lk{C}$ are such that $C_{1} \ncol{C} C_{2}$, then $C_{1}$ and $C_{2}$ share exactly $g-2$ curves.
\end{Rem}
\begin{figure}[h]
\begin{center}
\includegraphics[width=12cm]{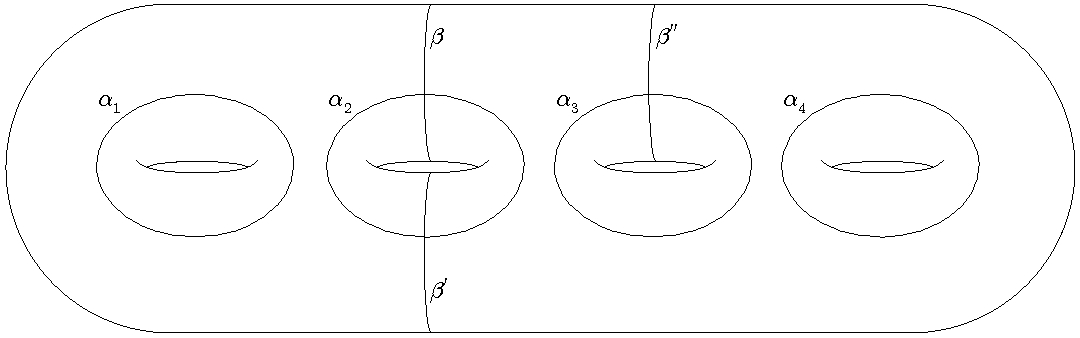}\caption{The cut systems $\{\alpha_{1},\beta,\alpha_{3},\alpha_{4}\}$ and $\{\alpha_{1},\beta^{\prime},\alpha_{3},\alpha_{4}\}$ are in the same colour with respect to $\{\alpha_{1},\ldots,\alpha_{4}\}$, while $\{\alpha_{1},\alpha_{2},\beta^{\biprime},\alpha_{4}\}$ is not.}\label{ExampleColoursFig}
\end{center}
\end{figure}
\indent Let $\gamma$ be a nonseparating curve of $S$. Following Irmak and Korkmaz's work on the Hatcher-Thurston complex (for which we recall $\htcomp{S}$ is the $1$-skeleton) in \cite{IrmakKorkmaz}, we define the graph $X_{\gamma}^{S}$ as the simplicial graph whose vertices are the nonseparating curves $\beta$ on $S$ such that $i(\beta,\gamma) = 1$, and two vertices $\alpha$ and $\beta$ span an edge if $i(\alpha,\beta) = 1$.\\
\indent In \cite{IrmakKorkmaz}, we obtain the following result, modifying the statement to suit the notation used here.
\begin{Lema}[\cite{IrmakKorkmaz}]\label{graphXgamma}
 Let $S = S_{g,n}$ such that  $g \geq 1$ and $n \geq 0$, and $\gamma$ be a nonseparating curve on $S$. Then $X_{\gamma}^{S}$ is connected.
\end{Lema}
\indent A \textit{triangle} on $\htcomp{S}$ is a set of three distinct cut systems on $S$, whose elements pairwise span edges in $\htcomp{S}$. Now we prove that for every triangle in $\htcomp{S}$ there exists a convenient multicurve contained in each cut system.
\begin{Lema}\label{multicurveintriangle}
 Let $S = S_{g,n}$ such that $g \geq 1$ and $n \geq 0$ punctures, and $T$ be a triangle on $\htcomp{S}$. Then, there exists a unique multicurve $M$, of cardinality $g-1$, such that $M$ is contained in every element of $T$.
\end{Lema}
\begin{proof}
 Let us denote $T = \{A,B,C\}$. Since $A,B \in \lk{C}$ then if $A \ncol{C} B$, by Remark \ref{g-2curves}, $|A \cap B| = g-2$; but then $A$ and $B$ would not be able to span an edge, contradicting $T$ being a triangle. Thus $A \col{C} B$. Since $A \neq B$ we have $M = A \cap C = B \cap C$ is the desired multicurve of cardinality $g-1$.
\end{proof}
\begin{Lema}\label{colourtriangles}
 Let $A, B, C$ be distinct cut systems on $S$, such that $A,B \in \lk{C}$. Then $A \col{C} B$ if and only if there exists a finite collection of triangles $T_{1}, \ldots, T_{m}$ such that $A, C \in T_{1}$, $B,C \in T_{m}$, and $T_{i}$ and $T_{i+1}$ share exactly one edge for $i = 1, \ldots, m-1$.
\end{Lema}
\begin{proof}
 If $g = 1$, then we obtain the desired result directly from Lemma \ref{graphXgamma}, making $C = \{\gamma\}$. So, suppose $g > 1$.\\
 \indent If $A \col{C} B$, let $M = A \cap C = B \cap C$ be the multicurve of Lemma \ref{multicurveintriangle} with cardinality $g-1$. Let $\alpha$ be the curve in $A \backslash M$, $\beta$ be the curve in $B \backslash M$ and $\gamma$ be the curve in $C \backslash M$. Since $A, B \in \lk{C}$ then $\alpha$ and $\gamma$ intersect once, just the same as $\beta$ and $\gamma$; moreover, $\alpha$, $\beta$ and $\gamma$ are nonseparating curves of $S \backslash M$ since $A$, $B$ and $C$ are cut systems. Thus $\alpha$ and $\beta$ are vertices in $X_{\gamma}^{S \backslash M}$, and by Lemma \ref{graphXgamma} there exists a finite collection of nonseparating (in $S \backslash M$) curves $c_{0}, \ldots, c_{m}$ with $\alpha = c_{0}$, $\beta = c_{m}$ and $c_{i}$ adjacent to $c_{i+1}$ in $X_{\gamma}^{S \backslash M}$. Since every $c_{i}$ is a nonseparating curve of $S \backslash M$, then $\{c_{i}\} \cup M$ is a cut system of $S$ for each $i$; in particular $A = \{c_{0}\} \cup M$ and $B = \{c_{m}\} \cup M$. By construction, $T_{i+1} \ColonEqq \{\{c_{i}\} \cup M, C, \{c_{i+1}\} \cup M\}$ is a triangle for $i = 0, \ldots, m-1$, $T_{i}$ and $T_{i+1}$ share exactly one edge for $i = 0, \ldots, m-1$, $A,C \in T_{1}$, and $C,B \in T_{m}$.\\
 \indent Conversely, if $T_{1}, \ldots, T_{m}$ is a finite collection of triangles such that $A,C \in T_{1}$, $B,C \in T_{m}$ and $T_{i}$ and $T_{i+1}$ share exactly one edge for $i = 1, \ldots, m-1$, we denote by $M_{i}$ the multicurve corresponding to the triangle $T_{i}$ obtained by Lemma \ref{multicurveintriangle}, in particular $A \cap C = M_{1}$ and $B \cap C = M_{m}$. Let $D_{i}$ and $D_{i}^{\prime}$ be the cut systems in the triangle that span the edge shared by $T_{i}$ and $T_{i+1}$; since $D_{i} \cap D_{i}^{\prime} = M_{i}$ in $T_{i}$ and $D_{i} \cap D_{i}^{\prime} = M_{i+1}$ in $T_{i+1}$, we have that $M_{i} = M_{i+1}$ for $i = 1, \ldots, m-1$. Thus $M_{i} = M_{j}$ for $i \neq j$, so $A \cap C = B \cap C$ which by definition implies that $A \col{C} B$.
\end{proof}
\section{Proof of Theorem \ref{TheoA}}\label{chap4sec2}
\indent In this section, let all surfaces be of genus at least $2$, possibly with punctures.\\
\indent An \textit{alternating square} in $\htcomp{S}$ is a closed path with four distinct consecutive vertices $C_{1}$, $C_{2}$, $C_{3}$, $C_{4}$ such that $C_{1} \ncol{C_{2}} C_{3}$ and $C_{2} \ncol{C_{3}} C_{4}$. So, $C_{1}$ and $C_{3}$ have exactly $g-2$ curves in common, and $C_{2}$ and $C_{4}$ have also exactly $g-2$ curves in common. In Figure \ref{ExampleColoursFig} the cut systems $\{\alpha_{1},\alpha_{2},\alpha_{3},\alpha_{4}\}$, $\{\alpha_{1},\beta,\alpha_{3},\alpha_{4}\}$, $\{\alpha_{1},\beta,\beta^{\biprime},\alpha_{4}\}$ and $\{\alpha_{1},\alpha_{2},\beta^{\biprime},\alpha_{2}\}$ form an alternating square.
\begin{Lema}\label{Msubsetaltsquare}
 Let $C_{1}$, $C_{2}$, $C_{3}$, $C_{4}$ be consecutive vertices of an alternating square in $\htcomp{S}$. Then $C_{1} \cap C_{2} \cap C_{3} \cap C_{4}$ has cardinality $g-2$.
\end{Lema}
\begin{proof}
 Since $C_{1},C_{3} \in \lk{C_{2}}$,$C_{1} \cap C_{3} = C_{1} \cap C_{2} \cap C_{3}$; analogously $C_{1} \cap C_{3} = C_{1} \cap C_{4} \cap C_{3}$. This implies that $C_{1} \cap C_{2} \cap C_{3} \subset C_{4}$, thus $C_{1} \cap C_{2} \cap C_{3} \cap C_{4} = C_{1} \cap C_{3}$. Given that $C_{1} \ncol{C_{2}} C_{3}$, we have $|C_{1} \cap C_{2} \cap C_{3} \cap C_{4}| = g-2$.
\end{proof}
\begin{Lema}\label{consecutivealternating}
 Let $C_{1},C_{2},C_{3}$ be cut systems on $S$, such that $C_{1},C_{3} \in \lk{C_{2}}$ and $C_{1} \ncol{C_{2}} C_{3}$. There exists $C_{3}^{\prime} \in \lk{C_{2}}$ with $C_{3}^{\prime} \col{C_{2}} C_{3}$, such that $C_{1}$, $C_{2}$, $C_{3}^{\prime}$ are consecutive vertices of an alternating square.
\end{Lema}
\begin{proof}
 Since $C_{1} \ncol{C_{2}} C_{3}$, let $M$ be the common multicurve of $C_{1}$, $C_{2}$ and $C_{3}$ obtained by Lemma \ref{multicurveintriangle}. Let also $\alpha,\beta,\alpha^{\prime},\beta^{\prime}$ be the curves such that $C_{1} = M \cup \{\alpha^{\prime},\beta\}$, $C_{2} = M \cup \{\alpha,\beta\}$ and $C_{3} = M \cup \{\alpha,\beta^{\prime}\}$.\\
 \indent Let $T$ be a regular neighbourhood of $\{\alpha, \alpha^{\prime}\}$. Since $i(\alpha,\alpha^{\prime}) = 1$, $T$ is homeomorphic to $S_{1,1}$. Let $\beta^{\biprime}$ be a nonseparating curve of $S \backslash M$ such that $i(\beta,\beta^{\biprime}) = 1$, and $\beta^{\biprime}$ is contained in $S \backslash T$ (that is possible since $S \backslash M$ has genus $2$). By construction we have the following: $C_{3}^{\prime} = M \cup \{\alpha,\beta^{\biprime}\}$ and $C_{4} = M \cup \{ \alpha^{\prime},\beta^{\biprime}\}$ are cut systems such that  $C_{3}^{\prime} \in \lk{C_{2}} \cap \lk{C_{4}}$, $C_{4} \in \lk{C_{1}} \cap \lk{C_{3}}$, $C_{3} \col{C_{2}} C_{3}^{\prime}$, $C_{1} \ncol{C_{2}} C_{3}^{\prime}$ and $C_{2} \ncol{C_{3}^{\prime}} C_{4}$. Thus $C_{1}$, $C_{2}$, $C_{3}^{\prime}$, $C_{4}$ are the consecutive vertices of an alternating square.
\end{proof}
\indent Let $S_{1} = S_{g_{1},n_{1}}$ and $S_{2} = S_{g_{2},n_{2}}$ with genus $g_{1} \geq 2$, $g_{2} \geq 1$ and $n_{1},n_{2} \geq 0$.\\
\indent A simplicial map $\fun{\phi}{\htcomp{S_{1}}}{\htcomp{S_{2}}}$ is said to be \textit{edge-preserving} if whenever $C_{1}$ and $C_{2}$ are two distinct cut systems that span an edge in $\htcomp{S_{1}}$, their images under $\phi$ are distinct and span an edge in $\htcomp{S_{2}}$.
\begin{Rem}\label{trianglestotriangles}
 Note that if $\fun{\phi}{\htcomp{S_{1}}}{\htcomp{S_{2}}}$ is an edge-preserving map, then triangles are mapped to triangles.
\end{Rem}
\indent The map $\phi$ is said to be \textit{alternating} if for all cut systems $C$ on $S_{1}$ and all $C_{1}, C_{2} \in \lk{C}$ such that $C_{1}$ and $C_{2}$ differ by exactly two curves, then $\phi(C_{1})$ and $\phi(C_{2})$ differ by exactly two curves. Note that this condition says nothing about $\phi(C)$ and its relation with $\phi(C_{1})$ and $\phi(C_{2})$.
\begin{Lema}\label{colourepa}
 Let $\fun{\phi}{\htcomp{S_{1}}}{\htcomp{S_{2}}}$ be an edge-preserving map, and $C_{1}$, $C_{2}$ and $C_{3}$ be cut systems on $S_{1}$ with $C_{1}, C_{3} \in \lk{C_{2}}$. If $C_{1} \col{C_{2}} C_{3}$, then $\phi(C_{1}) \col{\phi(C_{2})} \phi(C_{3})$. If $\phi$ is also alternating, then $C_{1} \ncol{C_{2}} C_{3}$ implies $\phi(C_{1}) \ncol{\phi(C_{2})} \phi(C_{3})$; in particular alternating squares go to alternating squares.
\end{Lema}
\begin{proof}
 If $C_{1} \col{C_{2}} C_{3}$, then by Lemma \ref{colourtriangles} there exists a finite collection of triangles $T_{1}, \ldots, T_{m}$ with $C_{1}, C_{2} \in T_{1}$, $C_{2}, C_{3} \in T_{m}$, and $T_{i}, T_{i+1}$ share one edge. By Remark \ref{trianglestotriangles} $\phi(T_{i})$ is a triangle for all $i = 1, \ldots, m$, with $\phi(C_{1}), \phi(C_{2}) \in \phi(T_{1})$, $\phi(C_{2}), \phi(C_{3}) \in \phi(T_{m})$ and $\phi(T_{i}), \phi(T_{i+1})$ sharing one edge; thus, once again by Lemma \ref{colourtriangles}, $\phi(C_{1}) \col{\phi(C_{2})} \phi(C_{3})$.\\
 \indent Let $\phi$ be also alternating, and $C_{1} \ncol{C_{2}} C_{3}$. By Remark \ref{g-2curves} $C_{1}, C_{3}$ differ by exactly $2$ curves and since $\phi$ is an \epa map, we have that $\phi(C_{1}), \phi(C_{3}) \in N(\phi(C_{2}))$ and $\phi(C_{1}), \phi(C_{3})$ differ by exactly $2$ curves; so, $\phi(C_{1}) \ncol{\phi(C_{2})} \phi(C_{3})$.\\
 \indent Let $\phi$ be an \epa map, and $\mathcal{S}$ be an alternating square with consecutive vertices $C_{1}, \ldots, C_{4}$. Since $C_{1}, C_{3} \in \lk{C_{2}} \cap \lk{C_{4}}$ then $\phi(C_{1}), \phi(C_{3}) \in N(\phi(C_{2})) \cap N(\phi(C_{4}))$, and as proved above $\phi(C_{1}) \ncol{\phi(C_{2})} \phi(C_{3})$ and $\phi(C_{2}) \ncol{\phi(C_{3})} \phi(C_{4})$ (since $C_{1} \ncol{C_{2}} C_{3}$ and $C_{2} \ncol{C_{3}} C_{4}$). Therefore $\phi(\mathcal{S})$ is an alternating square.
\end{proof}
\indent Note that this lemma allows us to see the importance of the alternating requirement for $\phi$. If $\phi$ were only edge-preserving (or locally injective), we would not have enough information to be certain that alternating squares are mapped to alternating squares, which is an important requirement if we ever want $\phi$ to be induced by a homeomorphism. Moreover, the rest of the results presented here would be much more complicated to prove if at all possible.\\[0.3cm]
\indent Now we are ready to prove Theorem \ref{TheoA} (which is quite similar to a result about locally injective maps for the Pants complex, that appears as Theorem C in \cite{AraPants}, though we must note that for some of the arguments in the proof being an alternating map is a key requirement).
\begin{proof}[\textbf{Proof of Theorem \ref{TheoA}}]
\indent Let $A$ be a vertex of $\htcomp{S_{1}}$. Then let $\{B_{1}, \ldots, B_{g_{1}}\} \subset \lk{A}$ be a set of representatives for the colours of $\lk{A}$. Since $B_{i} \col{A} B_{j}$ if and only if $i=j$ then $\phi(B_{i}) \col{\phi(A)} \phi(B_{j})$ if and only if $i=j$, by Lemma \ref{colourepa}; thus $\lk{\phi(A)}$ has at least a many colours as $\lk{A}$, so $g_{1} \leq g_{2}$.\\
\indent This implies that $M = \phi(B_{1}) \cap \cdots \cap \phi(B_{g_{1}})$ has cardinality $g_{2} - g_{1}$. We must also note that $M \subset \phi(A)$. We can easily check that if $B \col{A} B_{i}$ for some $i$, then $M \subset \phi(B)$: by Lemma \ref{colourepa} $\phi(B) \col{\phi(A)} \phi(B_{i})$, which means they were obtained from $\phi(A)$ by replacing the same curve, so $\phi(A) \cap \phi(B_{i}) \subset \phi(B)$; since $M \subset \phi(A) \cap \phi(B_{i})$, then $M \subset \phi(B)$.\\
\indent With this we have proved that for all $B \in A \cup \lk{A}$, $M \subset \phi(B)$. Given that $\htcomp{S_{1}}$ is connected, we only need to prove that given any element $B \in \lk{A}$, for all $C \in \lk{B}$, we have $M \subset \phi(C)$. Let $B$ and $C$ be such cut systems.\\
\indent If $C \col{B} A$, then by Lemma \ref{colourepa} $\phi(C) \col{\phi(B)} \phi(A)$, which means $\phi(C)$ and $\phi(A)$ were obtained by replacing the same curve of $\phi(B)$, so $\phi(C) \cap \phi(B) = \phi(A) \cap \phi(B)$. Since we have already proved that $M \subset \phi(B)$ and $M \subset \phi(A)$ then $M \subset \phi(A) \cap \phi(B) = \phi(C) \cap \phi(B)$. Thus $M \subset \phi(C)$.\\
\indent If $C \ncol{B} A$, then by Lemma \ref{consecutivealternating} there exists $C^{\prime} \in \lk{B}$ with $C \col{B} C^{\prime}$ and $C^{\prime} \ncol{B} A$, such that $A, B, C^{\prime}$ are consecutive vertices of an alternating square $\Sigma$. By Lemma \ref{colourepa} $\phi(\Sigma)$ is also an alternating square. Let $D$ the vertex of $\Sigma$ different from $A$, $B$ and $C^{\prime}$; since $D \in \lk{A}$, we have proved above that $M \subset \phi(D)$, thus $M \subset \phi(A) \cap \phi(B) \cap \phi(D)$ and as we have seen in the proof of Lemma \ref{Msubsetaltsquare}, $\phi(A) \cap \phi(B) \cap \phi(C^{\prime}) \cap \phi(D) = \phi(A) \cap \phi(B) \cap \phi(D)$, so $M \subset \phi(C^{\prime})$. Given that $C \col{B} C^{\prime}$, this leaves us in the previous case, therefore $M \subset \phi(C)$.
\end{proof}
\section{Proof of Theorem \ref{TheoB}}\label{chap4sec3}
\indent Hereinafter, let $S_{1} = S_{g,0}$ and $S_{2} = S_{g,n}$ with $g \geq 3$ and $n \geq 0$. Before giving the idea of the proof, we need the following definitions.\\
\indent We define the \textit{complexity of} $S_{g,n}$, denoted by $\kappa(S_{g,n})$ as $3g-3+n$. Note this is equal to the cardinality of a maximal multicurve.\\
\indent If $S_{g,n}$ is such that $\kappa(S_{g,n}) > 1$, the \textit{curve graph} $\ccomp{S_{g,n}}$, introduced by Harvey in \cite{Harvey}, is the simplicial graph whose vertices correspond to the curves of $S$, and two vertices span an edge if they are disjoint. We denote $\mathcal{V}(\ccomp{S_{g,n}})$ the set of vertices of $\ccomp{S_{g,n}}$.\\
\indent If $S_{g,n}$ is such that $g \geq 1$, the \textit{Schmutz graph} $\gcomp{S_{g,n}}$, introduced by Schmutz-Schaller in \cite{Schmutz}, is the simplicial graph whose vertices correspond to nonseparating curves of $S$, and where two vertices span an edge if they intersect once. We denote by $\mathcal{V}(\gcomp{S})$ the set of vertices of $\gcomp{S}$.\\[0.3cm]
\textbf{Idea of the proof:} We proceed by using $\phi$ to induce a map $\fun{\psi}{\mathcal{V}(\gcomp{S_{1}})}{\mathcal{V}(\gcomp{S_{2}})}$ in such a way that $\phi(\{\alpha_{1}, \ldots, \alpha_{g}\}) = \{\psi(\alpha_{1}), \ldots, \psi(\alpha_{g})\}$. Then we induce two maps $\fun{\widetilde{\phi}}{\htcomp{S_{1}}}{\htcomp{S_{1}}}$ and $\fun{\widetilde{\psi}}{\mathcal{V}(\gcomp{S_{1}})}{\mathcal{V}(\gcomp{S_{1}})}$ by filling the punctures of $S_{2}$. These maps also verify that $\widetilde{\phi}(\{\alpha_{1}, \ldots, \alpha_{g}\}) = \{\widetilde{\psi}(\alpha_{1}), \ldots, \widetilde{\psi}(\alpha_{g})\}$. Following the proofs of several properties of $\psi$ and $\widetilde{\psi}$, we extend $\widetilde{\psi}$ to an edge-preserving map $\fun{\widehat{\psi}}{\ccomp{S_{1}}}{\ccomp{S_{1}}}$ which, by Theorem A in \cite{JHH2}, is induced by a homeomorphism of $S_{1}$. Therefore $\widetilde{\psi}$ is induced by a homeomorphism of $S_{1}$.
\subsection{Inducing $\fun{\psi}{\gcomp{S_{1}}}{\gcomp{S_{2}}}$ and $\fun{\widetilde{\psi}}{\gcomp{S_{1}}}{\gcomp{S_{1}}}$}\label{chap4sec3subsec1}
\indent Let $\alpha$ be a nonseparating curve. Recall that $\htstrat{\{\alpha\}}{S_{1}}$ is isomorphic to $\htcomp{S_{1} \backslash \{\alpha\}}$. Then, given an \epa map $\fun{\phi}{\htcomp{S_{1}}}{\htcomp{S_{2}}}$ we can obtain an \epa map $\fun{\phi_{\alpha}}{\htcomp{S_{1} \backslash \{\alpha\}}}{\htcomp{S_{2}}}$. Applying Theorem \ref{TheoA} to $\phi_{\alpha}$ we know there exists a unique multicurve on $S_{2}$ of cardinality $1$, contained in the image under $\phi$ of every cut system containing $\alpha$; we will denote the element of this multicurve as $\psi(\alpha)$. In this way we have defined a function $\fun{\psi}{\mathcal{V}(\gcomp{S_{1}})}{\mathcal{V}(\gcomp{S_{2}})}$.
\begin{Lema}\label{lemadefpsi}
 Let $\fun{\phi}{\htcomp{S_{1}}}{\htcomp{S_{2}}}$ be an \epa map and $\fun{\psi}{\mathcal{V}(\gcomp{S_{1}})}{\mathcal{V}(\gcomp{S_{2}})}$ be the induced map on the nonseparating curves. If $\alpha$ and $\beta$ are nonseparating curves and $C$ a cut system on $S_{1}$, then:
 \begin{enumerate}
  \item If $\alpha \in C$, then $\psi(\alpha) \in \phi(C)$.
  \item If $\alpha \neq \beta$ and $\alpha, \beta \in C$, then $\psi(\alpha) \neq \psi(\beta)$.
  \item If $i(\alpha,\beta) = 1$, then $i(\psi(\alpha),\psi(\beta)) = 1$.
 \end{enumerate}
\end{Lema}
\begin{proof}
 (1) Follows directly from the definition.\\
 \indent (2) Let $C = \{\alpha, \beta, \gamma_{1}, \ldots, \gamma_{g-2}\}$ and let $C_{\alpha}, C_{\beta}, C_{\gamma_{1}}, \ldots, C_{\gamma_{g-2}}$ be representatives of the colours in $\lk{C}$ indexed by $\alpha, \beta, \gamma_{1}, \ldots, \gamma_{g-2}$ respectively so that $\alpha \notin C_{\alpha}$, $\alpha \in C_{\beta}, C_{\gamma_{1}}, \ldots, C_{\gamma_{g-2}}$, $\beta \notin C_{\beta}$ and $\beta \in C_{\alpha}, C_{\gamma_{1}}, \ldots, C_{\gamma_{g-2}}$. Using Lemma \ref{colourepa} we have that $\phi(C_{\alpha})$, $\phi(C_{\beta})$, $\phi(C_{\gamma_{1}})$, $\ldots$, $\phi(C_{\gamma_{g-2}})$ are representatives of all the colours of $\lk{\phi(C)}$. By (1) we have that $\psi(\alpha) \in \phi(C) \cap \phi(C_{\beta}) \cap \phi(C_{\gamma_{1}}) \cap \ldots \cap \phi(C_{\gamma_{g-2}})$, so $\psi(\alpha)$ cannot be an element of $\phi(C_{\alpha})$ and, since $\beta \in C_{\alpha}$, by (1) again we have that $\psi(\beta) \in \phi(C_{\alpha})$. Therefore $\psi(\alpha) \neq \psi(\beta)$.\\
 \indent (3) Using a regular neighbourhood of $\{\alpha, \beta\}$, we can find a multicurve $M$ in $S_{1}$ such that $C^{\prime} = \{\alpha\} \cup M$ and $C^{\biprime} = \{\beta\} \cup M$ are cut systems; this implies that if $C^{\prime}$ and $C^{\biprime}$ span an edge in $\htcomp{S_{1}}$, then $\phi(C^{\prime})$ and $\phi(C^{\biprime})$ span an edge in $\htcomp{S_{2}}$. By (1) and (2), $\phi(C^{\prime}) = \{\psi(\alpha)\} \cup \psi(M)$ and $\phi(C^{\biprime}) = \{\psi(\beta)\} \cup \psi(M)$, therefore $i(\psi(\alpha), \phi(\beta)) = 1$.
\end{proof}
\indent Note that this lemma implies that if $C = \{\alpha_{1}, \ldots, \alpha_{g}\}$, we have that $\phi(C) = \{\psi(\alpha_{1}), \ldots, \phi(\alpha_{g})\}$.\\
\indent By filling the punctures of $S_{2}$ and identifying the resulting surface with $S_{1}$, we obtain a map $\fun{\piC}{\mathcal{V}(Y(S_{2}))}{\mathcal{V}(\ccomp{S_{1}})}$, where $Y(S_{2})$ is the subcomplex of $\ccomp{S_{2}}$ whose vertices correspond to curves $\gamma$ on $S_{2}$ such that all the connected components of $S_{2} \backslash \{\gamma\}$ have positive genus. Observe that $\piC$ sends nonseparating curves of $S_{2}$ into nonseparating curves of $S_{1}$, and separating curves of $S_{2}$ that separate the surface in connected components of genus $g^{\prime} > 0$ and $g^{\biprime} > 0$ into separating curves of $S_{1}$ that separate the surface in connected components of genus $g^{\prime}$ and $g^{\biprime}$. In particular, if $C$ is a cut system, $\piC(C)$ is also a cut system, thus we obtain a map $\fun{\piHT}{\mathcal{V}(\htcomp{S_{2}})}{\mathcal{V}(\htcomp{S_{1}})}$.\\
\indent Now, from $\fun{\phi}{\htcomp{S_{1}}}{\htcomp{S_{2}}}$ we can obtain the map $$\fun{\widetilde{\psi}\ColonEqq \piC \circ \psi}{\mathcal{V}(\gcomp{S_{1}})}{\mathcal{V}(\gcomp{S_{1}})},$$ and the map $$\fun{\widetilde{\phi}\ColonEqq \piHT \circ \phi}{\mathcal{V}(\htcomp{S_{1}})}{\mathcal{V}(\htcomp{S_{1}})}.$$
\begin{Cor}\label{cordefpsitilde}
 Let $\fun{\phi}{\htcomp{S_{1}}}{\htcomp{S_{2}}}$ an \epa map, $\fun{\psi}{\mathcal{V}(\gcomp{S_{1}})}{\mathcal{V}(\gcomp{S_{2}})}$ be the induced map on the nonseparating curves, and $\widetilde{\phi}$ and $\widetilde{\psi}$ as above. If $\alpha$ and $\beta$ are nonseparating curves and $C$ a cut system on $S_{1}$, then:
 \begin{enumerate}
  \item If $\alpha \in C$ then $\widetilde{\psi}(\alpha) \in \widetilde{\phi}(C)$.
  \item If $\alpha \neq \beta$ and $\alpha, \beta \in C$ then $\widetilde{\psi}(\alpha) \neq \widetilde{\psi}(\beta)$.
  \item If $i(\alpha, \beta) = 1$ then $i(\widetilde{\psi}(\alpha), \widetilde{\psi}(\beta)) = 1$.
 \end{enumerate}
\end{Cor}
\begin{proof}
 (1) Follows from Lemma \ref{lemadefpsi}.\\
 \indent (2) If $\alpha \neq \beta$ and $\alpha, \beta \in C$ then by Lemma \ref{lemadefpsi} $\psi(\alpha), \psi(\beta) \in \phi(C)$ and $\psi(\alpha) \neq \psi(\beta)$. This implies that $\psi(\alpha)$ and $\psi(\beta)$ are disjoint curves that do not together separate $S_{2}$; these two properties together are preserved by $\piC$. Indeed, let $S^{\prime}$ be a subsurface of $S_{2}$ such that $\psi(\alpha), \psi(\beta) \in \ccomp{S^{\prime}}$ and $S^{\prime}$ is homeomorphic to $S_{2,1}$; let $\gamma$ be the boundary curve of $S_{2} \backslash \mathrm{int}(S^{\prime})$, then $\gamma$ separates $S_{2}$ in two connected components, each of positive genus. Thus $S^{\prime}$ is unaffected by $\piC$, i.e. $\piC|_{\mathcal{V}(\ccomp{S^{\prime}})} = \mathrm{id}_{\mathcal{V}(\ccomp{S^{\prime}})}$. Therefore $\widetilde{\psi}(\alpha) \neq \widetilde{\psi}(\beta)$.\\
 \indent (3) Since $i(\psi(\alpha), \psi(\beta)) = 1$, let $T$ be a regular neighbourhood of $\{\psi(\alpha),\psi(\beta)\}$. Then $T$ is homeomorphic to $S_{1,1}$. Let $\gamma$ be the boundary curve in $S_{2} \backslash \mathrm{int}(T)$; then $\gamma$ is a separating curve that separates $S_{2}$ in two connected components, each of positive genus. Thus, as in (2), $T$ is unaffected by $\piC$. Therefore $i(\widetilde{\psi}(\alpha), \widetilde{\psi}(\beta)) = 1$.
\end{proof}
\indent Similarly to Lemma \ref{lemadefpsi}, this implies that if $C = \{\alpha_{1}, \ldots, \alpha_{g}\}$, we have that $\widetilde{\phi}(C) = \{\widetilde{\psi}(\alpha_{1}), \ldots, \widetilde{\psi}(\alpha_{g})\}$.\\
\indent As a consequence of Lemma \ref{lemadefpsi} and Corollary \ref{cordefpsitilde}, we have that the maps $\psi$, $\widetilde{\psi}$ and $\widetilde{\phi}$ are simplicial. Moreover, we have the following result.
\begin{Cor}\label{cortildeedgepreserving}
 $\fun{\psi}{\gcomp{S_{1}}}{\gcomp{S_{2}}}$, $\fun{\widetilde{\psi}}{\gcomp{S_{1}}}{\gcomp{S_{1}}}$ and $\fun{\widetilde{\phi}}{\htcomp{S_{1}}}{\htcomp{S_{1}}}$ are edge-preserving maps. Also, $\widetilde{\phi}$ is an alternating map.
\end{Cor}
\indent A \textit{pants decomposition} of $S_{i}$ (for $i = 1,2$) is a maximal multicurve of $S_{i}$, i.e. it is a maximal complete subgraph of $\ccomp{S_{i}}$. Note that any pants decomposition of $S_{i}$ has exactly $\kappa(S_{i})$ curves.\\
\indent On the other hand, we say $P$ is a \textit{punctured pants decomposition} of $S_{2}$ if $\piC(P)$ is a pants decomposition of $S_{1}$. This implies that $S_{2} \backslash P$ is the disjoint union of $3g -3$ surfaces, with each connected component $P_{i}$ homeomorphic to $S_{0,3+k_{i}}$ such that $\sum_{i} k_{i} = n$.
\begin{Lema}\label{pantstopants1}
 Let $P$ be a pants decomposition of $S_{1}$ such that no two curves of $P$ together separate $S_{1}$. Then $\psi(P)$ is a punctured pants decomposition of $S_{2}$ and $\widetilde{\psi}(P)$ is a pants decomposition of $S_{1}$.
\end{Lema}
\begin{proof}
 Since for any two distinct curves $\alpha, \beta \in P$ we can always find a cut system containing both of them, by Lemma \ref{lemadefpsi} and Corollary \ref{cordefpsitilde} we know that $\psi(\alpha)$ is disjoint from $\psi(\beta)$ and $\widetilde{\psi}(\alpha)$ is disjoint from $\widetilde{\psi}(\beta)$. Thus, both $\psi(P)$ and $\widetilde{\psi}(P)$ are multicurves of cardinality $3g -3$, which means $\widetilde{\psi}(P)$ is a pants decomposition; then, by definition, $\psi(P)$ is a punctured pants decomposition.
\end{proof}
\begin{figure}[h]
\begin{center}
 \includegraphics[width=6cm]{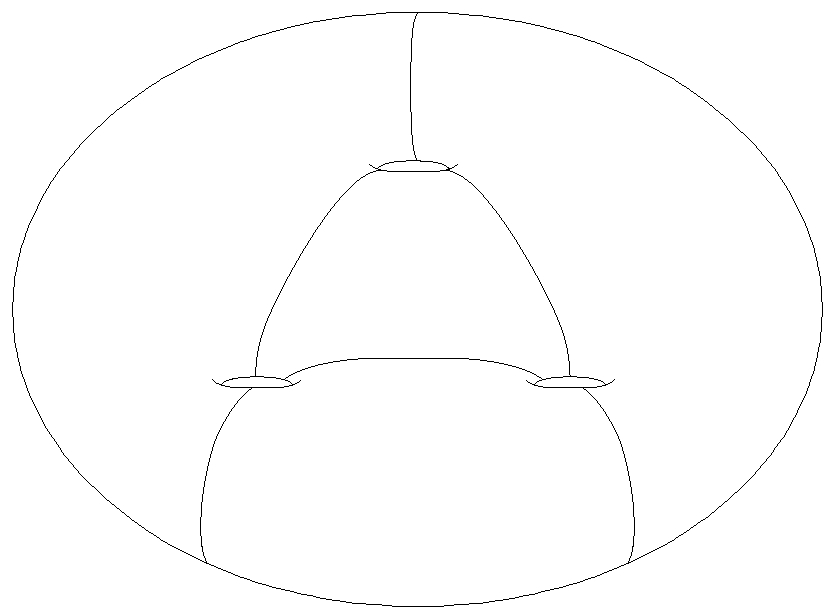} \hspace{1cm}\includegraphics[width=6cm]{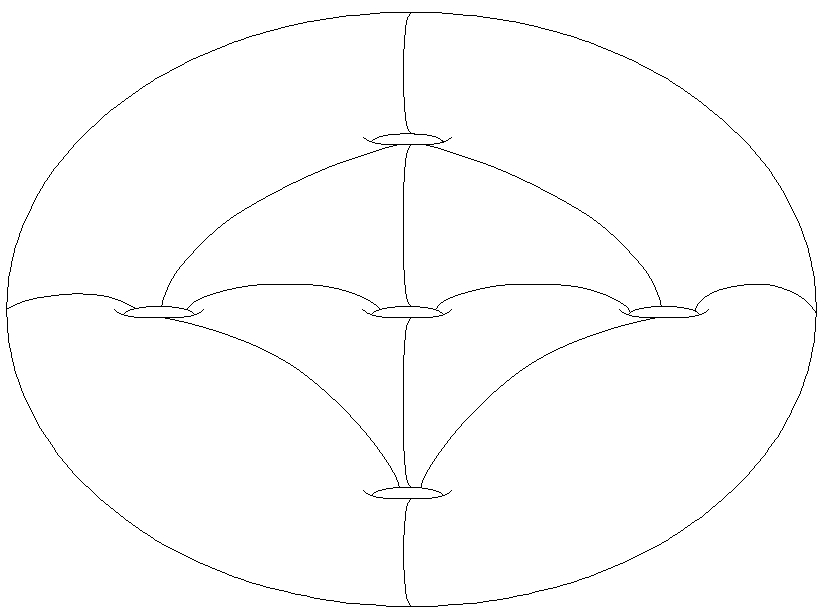} \caption{Pants decompositions for the closed surfaces of genus $3$ (left) and genus $5$ (right), such that no two curves of $P$ together separate.}\label{NonsepPantsDec}
\end{center}
\end{figure}
\indent The rest of this subsection consists of several technical definitions and lemmas, all of them leading to proving that both $\psi$ and $\widetilde{\psi}$ preserve disjointness and intersection number $1$, which we later use to extend their definitions to the respective curve complexes.\\[0.3cm]
\indent Let $\alpha$ and $\beta$ be two curve in $S_{1}$, and $N$ be a regular neighbourhood of $\{\alpha,\beta\}$. We say they are spherical-Farey neighbours if $N$ has genus zero and $i(\alpha,\beta) = 2$.\\
\indent Let $\alpha$ and $\beta$ be two nonseparating curves in $S_{1}$ that are spherical-Farey neighbours, and $N(\alpha,\beta)$ be their closed regular neighbourhood. Then $N(\alpha,\beta)$ is homeomorphic to a genus zero surface with four boundary components. Let $\veps_{0}$, $\veps_{1}$, $\veps_{2}$, $\veps_{3}$ be the boundary curves of $N(\alpha,\beta)$. We say $\veps_{i}$ and $\veps_{j}$ are connected outside of $N(\alpha,\beta)$, if there exists a proper arc in $S_{1} \backslash \mathrm{int}(N(\alpha,\beta))$ with one endpoint in $\veps_{i}$ and another in $\veps_{j}$.
\begin{Rem}\label{vepsconnected}
 If $\veps_{i}$ is a nonseparating curve, it has to be connected outside of $N(\alpha,\beta)$ to at least one other $\veps_{j}$ (with $i \neq j$), since otherwise there would not exist any curve intersecting $\veps_{i}$ exactly once, and thus $\veps_{i}$ would not be nonseparating.
\end{Rem}
\indent We say $\alpha$ and $\beta$ are \textit{of type A} if $\veps_{i}$ is a nonseparating curve for all $i$ and $\veps_{i}$ is connected outside of $\Sigma_{\alpha,\beta}$ to $\veps_{j}$ for all $i,j \in \{0,1,2,3\}$. See Figure \ref{ExampleTypeA}.
\begin{figure}[h]
\begin{center}
 \includegraphics[width=8cm]{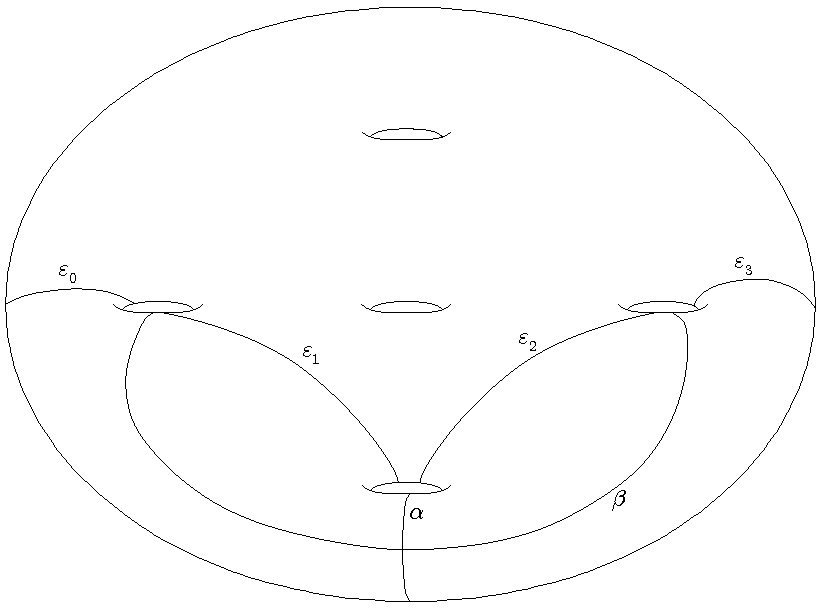} \caption{The curves $\alpha$ and $\beta$ are spherical-Farey neighbours of type A.}\label{ExampleTypeA}
\end{center}
\end{figure}
\begin{Rem}\label{pidisjoint}
 Remember that while $\piHT$ is an edge-preserving map it is not alternating. Also, $\piC$ has the property that if $\alpha$ and $\beta$ are disjoint nonseparating curves, then $i(\piC(\alpha),\piC(\beta)) = 0$, since forgetting the punctures only affects the connected components of $S \backslash \{\alpha,\beta\}$ by possibly transforming one of them into a cylinder.
\end{Rem}
\begin{Lema}\label{inter2}
 Let $\alpha$ and $\beta$ be two nonseparating curves in $S_{1}$ that are spherical-Farey neighbours of type A. Then $i(\psi(\alpha),\psi(\beta)) \neq 0 \neq i(\widetilde{\psi}(\alpha),\widetilde{\psi}(\beta))$.
\end{Lema}
\begin{proof} This proof is divided in three parts: the first proves that $\psi(\alpha) \neq \psi(\beta)$, the second proves that $\widetilde{\psi}(\alpha) \neq \widetilde{\psi}(\beta)$, and finally the third proves that $i(\psi(\alpha),\psi(\beta)) \neq 0 \neq i(\widetilde{\psi}(\alpha),\widetilde{\psi}(\beta))$.\\
\textit{First part:} Since $\alpha$ and $\beta$ are of type A, we can always find curves $\gamma$ and $\delta$ such that:
 \begin{itemize}
  \item $i(\alpha, \gamma) = i(\beta,\delta) = 1$.
  \item $i(\alpha, \delta) = i(\beta,\gamma) = i(\gamma,\delta) = 0$.
  \item There exists a multicurve $M$ of cardinality $g-2$ such that $C_{1} = \{\alpha,\delta\} \cup M$, $C_{2} = \{\beta,\gamma\} \cup M$ and $C_{0} = \{\gamma,\delta\} \cup M$ are cut systems.
 \end{itemize}
 See Figure \ref{ExampleSphericalFarey1} for a way to obtain them.
 \begin{figure}[h]
 \begin{center}
  \includegraphics[width=8cm]{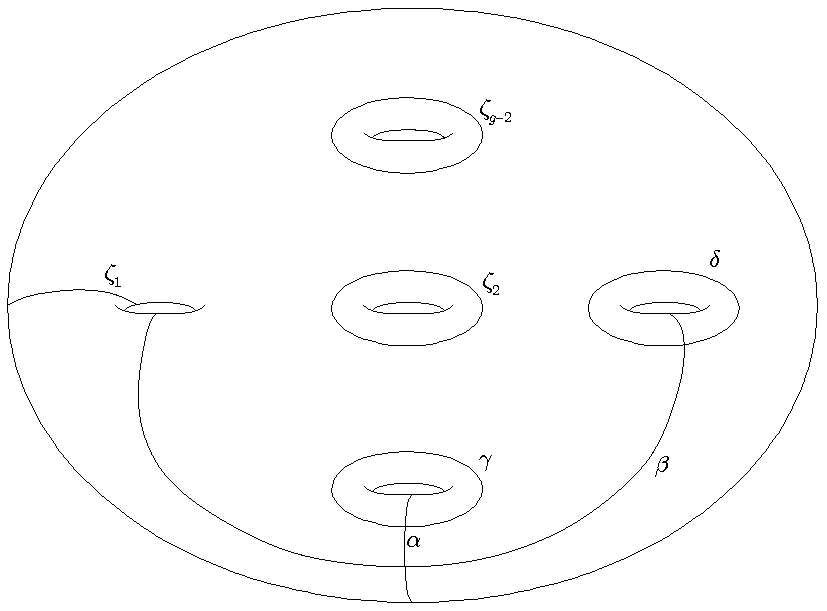}\caption{$M = \{\zeta_{1}, \ldots, \zeta_{g-2}\}$ is a multicurve such that $\{\alpha,\delta\} \cup M$, $\{\beta,\gamma\} \cup M$ and $\{\gamma,\delta\} \cup M$ are cut systems.}\label{ExampleSphericalFarey1}
 \end{center}
 \end{figure}\\
 \indent Then $C_{1}$, $C_{2} \in \lk{C_{0}}$ and $C_{1} \ncol{C_{0}} C_{2}$, so by definition and Lemma \ref{colourepa} $\phi(C_{1}), \phi(C_{2}) \in \lk{\phi(C_{0})}$ and $\phi(C_{1}) \ncol{\phi(C_{0})} \phi(C_{2})$; by Remark \ref{g-2curves} $\phi(C_{1})$ and $\phi(C_{2})$ share exactly $g-2$ curves, thus $\psi(\alpha) \neq \psi(\beta)$.\\
 \textit{Second part:} Using the cut systems $C_{1}$, $C_{2}$ and $C_{0}$ from the first part of this proof, we can then apply Corollary \ref{cordefpsitilde}, thus getting that $\widetilde{\psi}(\alpha)$ is disjoint from $\widetilde{\psi}(\delta)$ while $i(\widetilde{\psi}(\beta), \widetilde{\psi}(\delta)) = 1$. Then $\widetilde{\psi}(\alpha) \neq \widetilde{\psi}(\beta)$.\\
 \textit{Third part:} Let $\widetilde{P}$ be a multicurve such that $P_{1} = \widetilde{P} \cup \{\alpha\}$ and $P_{2} = \widetilde{P} \cup \{\beta\}$ are pants decompositions such that for $i = 1,2$, any two curves of $P_{i}$ do not separate the surface (see Figure \ref{ExampleSphericalFarey2-3} for an example). By Lemma \ref{pantstopants1} then $\widetilde{\psi}(P_{1})$ and $\widetilde{\psi}(P_{2})$ are pants decompositions of $S_{1}$ and, by the above paragraph, will differ in exactly one curve, $\widetilde{\psi}(\alpha)$ and $\widetilde{\psi}(\beta)$, meaning that they are contained in a complexity-one subsurface of $S_{1}$; given that by the second part of the proof, these two curves are different and yet they are contained in a subsurface of complexity one, we have that $i(\widetilde{\psi}(\alpha),\widetilde{\psi}(\beta)) \neq 0$. Given that $\widetilde{\psi} = \piC \circ \psi$, by Remark \ref{pidisjoint} we have that $i(\psi(\alpha),\psi(\beta)) \neq 0$.
\end{proof}
\begin{figure}[h]
\begin{center}
 \includegraphics[width=7cm]{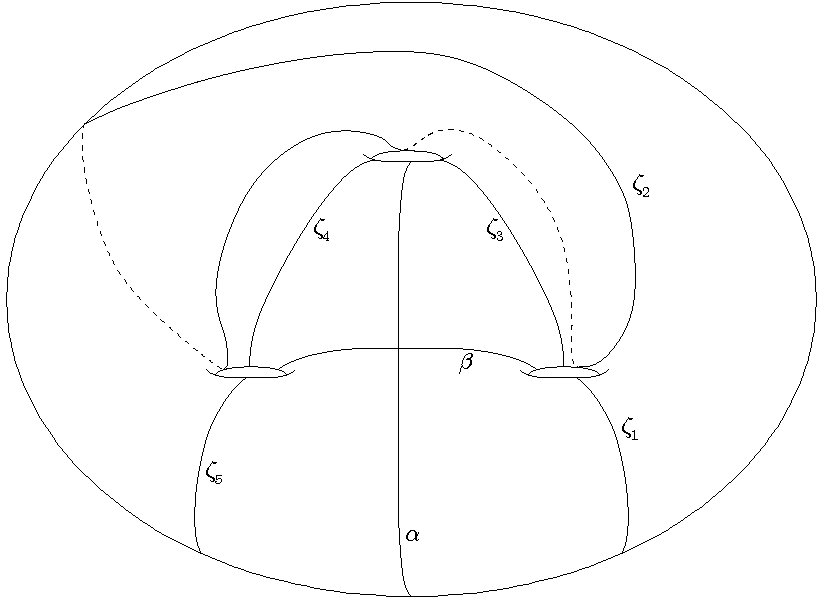} \hspace{1cm} \includegraphics[width=7cm]{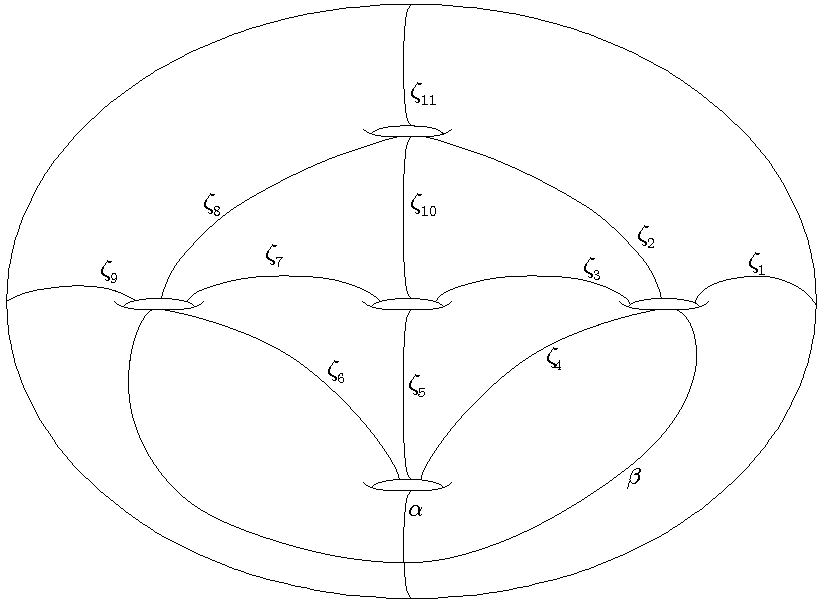} \caption{$\widetilde{P} = \{\zeta_{1},\ldots\}$ is a multicurve such that $P_{1} = \widetilde{P} \cup \{\alpha\}$ and $P_{2} = \widetilde{P} \cup \{\beta\}$ are pants decompositions such that for $i = 1,2$, any two curves of $P_{i}$ do not separate the surface.} \label{ExampleSphericalFarey2-3}
\end{center}
\end{figure}
\indent A \textit{halving multicurve} of a surface $S = S_{g,n}$ is a multicurve $H$ whose elements are nonseparating curves on $S$ such that: $S \backslash H = Q_{1} \sqcup Q_{2}$, with $Q_{1}$ and $Q_{2}$ homeomorphic to $S_{0,n_{1}}$ and $S_{0,n_{2}}$ respectively, and $n_{1} + n_{2} = 2(g+1) + n$. Note that a halving multicurve has exactly $g+1$ elements.\\
\indent We define a \textit{cutting halving multicurve} as a halving multicurve such that \textbf{any} $g$ elements of it form a cut system. Note that there exist halving multicurves that are not cutting halving multicurves, see Figure \ref{ExampleHalvingNoncutting} for an example.
\begin{figure}[h]
 \begin{center}
  \resizebox{12cm}{!}{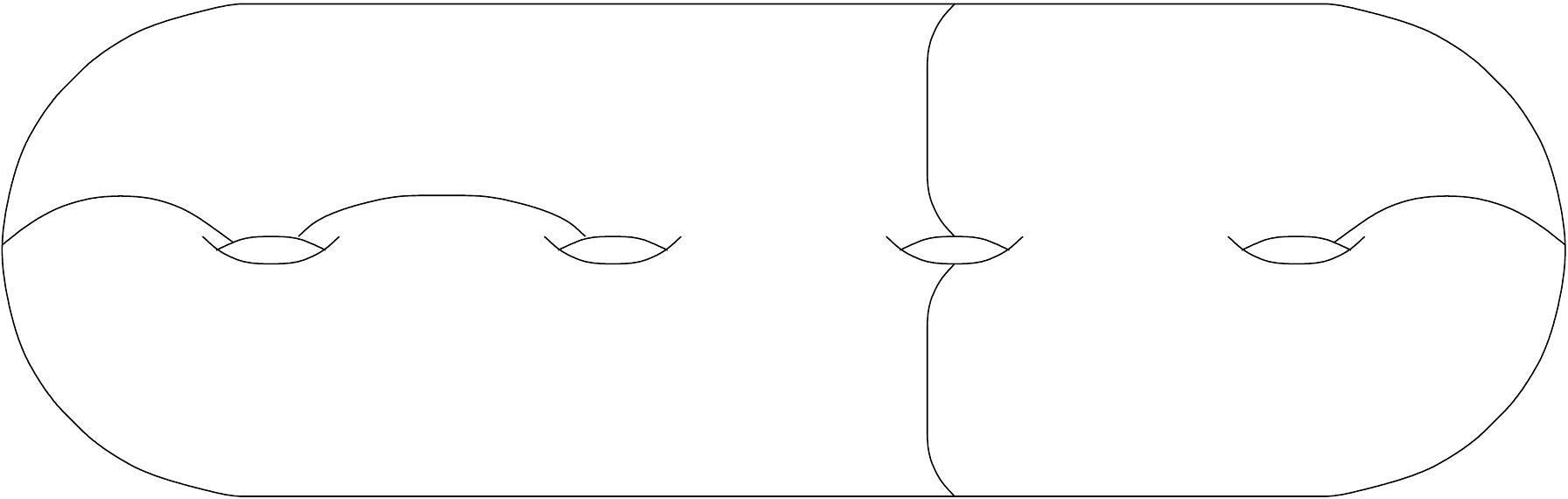} \caption{An example of a halving multicurve that is not a cutting halving multicurve.} \label{ExampleHalvingNoncutting}
 \end{center}
\end{figure}
\begin{Lema}\label{cuthalfmulti}
 If $H$ is a cutting halving multicurve of $S_{1}$, then $\psi(H)$ and $\widetilde{\psi}(H)$ are cutting halving multicurves of $S_{2}$ and $S_{1}$ respectively.
\end{Lema}
\begin{proof}
 Since $H$ is a cutting halving multicurve of $S_{1}$ then, by a repeated use of Lemma \ref{lemadefpsi} and Corollary \ref{cordefpsitilde}, $\psi(H)$ and $\widetilde{\psi}(H)$ will contain $g+1$ elements and any $g$ elements of $\psi(H)$ and $\widetilde{\psi}(H)$ will form cut systems. Therefore $S_{2} \backslash \psi(H)$ and $S_{1} \backslash \widetilde{\psi}(H)$ will have two connected components, each of genus zero; thus $\psi(H)$ and $\widetilde{\psi}(H)$ are cutting halving multicurves of $S_{2}$ and $S_{1}$ respectively.
\end{proof}
\begin{Lema}\label{inter0v2}
 Let $\alpha$ and $\beta$ be two disjoint nonseparating curves such that $S_{1} \backslash \{\alpha,\beta\}$ is disconnected. Then $\psi(\alpha)$ and $\psi(\beta)$ are disjoint in $S_{2}$ and $\widetilde{\psi}(\alpha)$ and $\widetilde{\psi}(\beta)$ are disjoint in $S_{1}$.
\end{Lema}
\begin{proof}
 \textit{We claim} $\psi(\alpha) \neq \psi(\beta)$ \textit{and} $\widetilde{\psi}(\alpha) \neq \widetilde{\psi}(\beta)$.\\
 \indent Given the conditions, let $\gamma$ be a nonseparating curve such that $\beta$ and $\gamma$ are spherical-Farey neighbours of type A, $\alpha$ and $\gamma$ are disjoint, and $S_{1} \backslash \{\alpha,\gamma\}$ is connected; then, by Lemmas \ref{lemadefpsi}, \ref{inter2} and Corollary \ref{cordefpsitilde}, $i(\psi(\alpha),\psi(\gamma)) = i(\widetilde{\psi}(\alpha),\widetilde{\psi}(\gamma)) = 0$ and $i(\psi(\beta),\psi(\gamma)) \neq 0 \neq i(\widetilde{\psi}(\beta),\widetilde{\psi}(\gamma))$. Therefore $\psi(\alpha) \neq \psi(\beta)$ and $\widetilde{\psi}(\alpha) \neq \widetilde{\psi}(\beta)$.\\[0.3cm]
 \textit{We claim} $i(\psi(\alpha),\psi(\beta)) = i(\widetilde{\psi}(\alpha),\widetilde{\psi}(\beta)) = 0$.\\
 \indent Let $H$ be a cutting halving multicurve in $S_{1}$ such that $\alpha$ is contained in $S_{1}^{\prime}$ and $\beta$ is contained in $S_{1}^{\biprime}$, where $S_{1}^{\prime}$ and $S_{1}^{\biprime}$ are the connected components of $S_{1} \backslash H$, and also such that $S_{1} \backslash \{\alpha,\gamma\}$ and $S_{1} \backslash \{\beta,\gamma\}$ are connected for all $\gamma \in H$. By Lemma \ref{cuthalfmulti} $\psi(H)$ is a cutting halving multicurve; let $S_{2}^{\prime}$ and $S_{2}^{\biprime}$ be the corresponding connected components of $S_{2} \backslash \psi(H)$. See Figure \ref{ExampleMulticurve1-2} for examples. By construction $\psi(\alpha)$ and $\psi(\beta)$ are disjoint from every element in $\psi(H)$, so they are curves contained in $S_{2} \backslash \psi(H)$.\\
 \indent If $\psi(\alpha)$ and $\psi(\beta)$ are in different connected components of $S_{2} \backslash \psi(H)$ then they are disjoint. So, suppose (without loss of generality) that both representatives are in $S_{2}^{\biprime}$.\\
 \indent Let $M$ be a multicurve of $S_{1}$ with the following properties.
 \begin{enumerate}
  \item Every element of $M$ is also a curve contained in $S_{1}^{\biprime}$ and $S_{1} \backslash \{\gamma,\delta\}$ is connected for all $\gamma,\delta \in M$.
  \item $S_{1} \backslash \{\gamma,\delta\}$ is connected for all $\gamma \in M$ and all $\delta \in H$.
  \item For all $\gamma \in M$, $\beta$ and $\gamma$ are spherical-Farey neighbours of type A.
  \item For all $\gamma \in M$, $S_{1} \backslash \{\alpha,\gamma\}$ is connected.
  \item $M$ has $g-2$ elements.
 \end{enumerate}
 \indent See Figure \ref{ExampleMulticurve1-2} for an example. By Lemma \ref{lemadefpsi}, $\psi(M)$ satisfy conditions 1, 2, 4 and 5; also, by Lemma \ref{inter2}, we have that for all $\gamma \in M$, $i(\psi(\beta),\psi(\gamma)) \neq 0$. This implies that every element of $\psi(M)$ is a curve contained in $S_{2}^{\biprime}$; thus $\psi(\gamma)$ intesects $\psi(\beta)$ at least twice for all $\gamma \in M$ (since $S_{2}^{\biprime}$ has genus zero, every curve contained in it is separating in $S_{2}^{\biprime}$).
 \begin{figure}[h]
 \begin{center}
  \includegraphics[height=4cm]{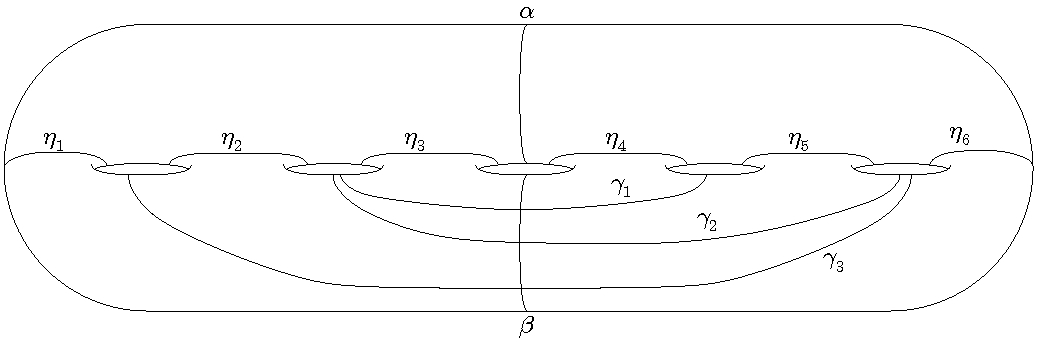}\\
  \includegraphics[height=4cm]{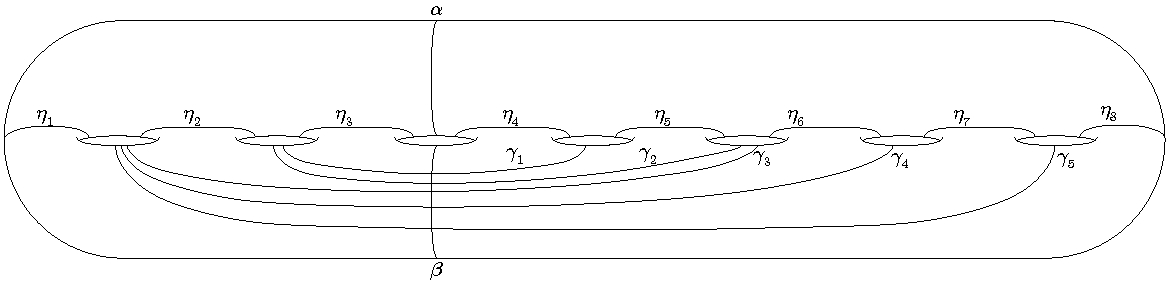} \caption{The cutting halving multicurve $H = \{\eta_{1}, \ldots, \eta_{g+1}\}$, the multicurve $M = \{\gamma_{1},\ldots,\gamma_{g-2}\}$, and the spherical-Farey neighbours $\alpha$ and $\beta$ for the closed surfaces of genus $5$ (above) and genus $7$ (below).}\label{ExampleMulticurve1-2}
 \end{center}
 \end{figure}\\
 \indent Let $U$ and $V$ be the connected components of $S_{2}^{\biprime} \backslash \{\psi(\alpha)\}$.\\
 \indent Now, we prove by contradiction that the elements of $\psi(M)$ are either all in $U$ or all in $V$: Let $\gamma, \gamma^{\prime} \in M$ be such that $\psi(\gamma)$ is contained in $U$ and $\psi(\gamma^{\prime})$ is contained in $V$. Then we can always find a curve $\delta$ contained in $S_{1}^{\biprime}$ such that the elements of $\{\gamma,\delta\}$ and of $\{\gamma^{\prime},\delta\}$ satisfy the conditions of Lemma \ref{inter2}, and such that $S_{1} \backslash \{\delta,\delta^{\prime}\}$ is connected for all $\delta^{\prime} \in H \cup \{\alpha\}$. This implies $i(\psi(\gamma),\psi(\delta)) \neq 0 \neq i(\psi(\gamma^{\prime}),\psi(\delta))$, and that $\psi(\delta)$ has to be either in $U$ or $V$. These two conditions together imply that $\psi(\delta)$ is contained in both $U$ and $V$, which is a contradiction.\\
 \indent Therefore, $\psi(M)$ consists of $g-2$ nonseparating curves, no two of which separate $S_{2}$, and (up to relabelling) all these nonseparating curves are disjointly contained in $U$. But $U$ can have at most $g-3$ nonseparating (in $S_{2}$) curves that no pair of which separates $S_{2}$ (this number is actually the greatest possible cardinality of a punctured pants decomposition of $U$); so we have found a contradiction and thus $\psi(\alpha)$ and $\psi(\beta)$ are in different connected components and then $i(\psi(\alpha),\psi(\beta)) = 0$.\\
 \indent By Remark \ref{pidisjoint}, since $i(\psi(\alpha),\psi(\beta)) = 0$, then $i(\widetilde{\psi}(\alpha),\widetilde{\psi}(\beta)) = 0$.
\end{proof}
Thus, by using Lemmas \ref{lemadefpsi}, \ref{inter0v2} and Corollary \ref{cordefpsitilde}, we obtain the following corollary.
\begin{Cor}\label{halfmultipants}
 $\psi$ and $\widetilde{\psi}$ preserve both disjointness and intersection $1$.
\end{Cor}
\subsection{Inducing $\fun{\widehat{\psi}}{\ccomp{S_{1}}}{\ccomp{S_{1}}}$}\label{chap4sec3subsec2}
\indent To extend $\widetilde{\psi}$, we proceed in the same way as Irmak in \cite{Irmak3}, using chains and the fact that every separating curve in $S_{1}$ is the boundary curve of a closed neighbourhood of a chain.\\
\indent Using Lemmas \ref{lemadefpsi}, \ref{inter0v2} and Corollary \ref{cordefpsitilde}, we obtain the following lemma.
\begin{Lema}\label{chainstochains}
 If $X$ is a chain of length $k$, then $\psi(X)$ and $\widetilde{\psi}(X)$ are chains of length $k$.
\end{Lema}
Since $S_{1}$ is a closed surface, then every separating curve $\alpha$ on $S_{1}$ can be characterized as the boundary curve of a closed regular neighbourhood of a chain $X_{\alpha}$. See Figure \ref{ExampleDefChain} for an example. We call $X_{\alpha}$ \textit{a defining chain of} $\alpha$. Recall that every defining chain of a separating curve always has even cardinality, $2k$, and its closed regular neighbourhood will then have genus $k$.
\begin{figure}[h]
\begin{center}
 \includegraphics[width=12cm]{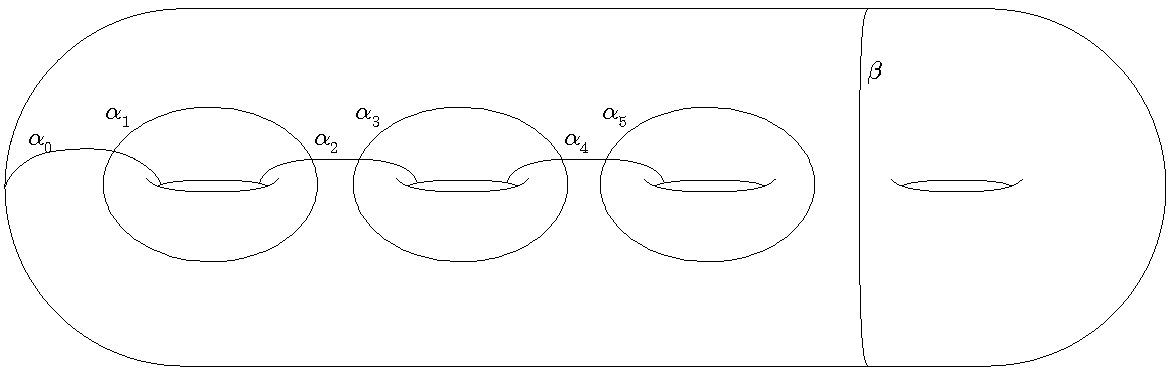} \caption{$\{\alpha_{0}, \ldots, \alpha_{5}\}$ is a defining chain of the separating curve $\beta$.} \label{ExampleDefChain}
\end{center}
\end{figure}
\begin{Lema}\label{differentchains}
 Let $\beta_{1}$ and $\beta_{2}$ be separating curves in $S_{1}$, and $X_{1}$ and $X_{2}$ be defining chains of $\beta_{1}$ and $\beta_{2}$ respectively. If $\beta_{1} = \beta_{2}$, then either every element of $X_{1}$ is disjoint from every element of $X_{2}$ and viceversa, or every curve in $X_{1}$ intersects at least one curve in $X_{2}$ and viceversa.
\end{Lema}
\begin{proof}
 Since every element in $X_{1}$ and $X_{2}$ is by definition disjoint from $\beta = \beta_{1} = \beta_{2}$, then all the elements in $X_{1}$ are contained in the same connected component of $S_{1} \backslash \{\beta\}$, and analogously with all the elements of $X_{2}$. If the elements of $X_{2}$ are in a different connected component from those of $X_{1}$ then every element of $X_{1}$ is disjoint from every element of $X_{2}$ and viceversa. If the elements of $X_{2}$ are in the same connected component as those of $X_{1}$, since $X_{1}$ fills its regular neighbourhood we have that every curve in $X_{1}$ intersects at least one curve in $X_{2}$ and viceversa.
\end{proof}
\indent To extend the definition of $\widetilde{\psi}$ to $\ccomp{S}$, we define $\widehat{\psi}$ as follows: If $\alpha$ is a nonseparating curve, then $\widehat{\psi}(\alpha) = \widetilde{\psi}(\alpha)$; if $\alpha$ is a separating curve, let $X_{\alpha}$ be a defining chain of $\alpha$ and then we define $\widehat{\psi}(\alpha)$ as the boundary curve of a regular neighbourhood of $\widetilde{\psi}(X_{\alpha})$. This makes sense given that the regular neighbourhoods of $X_{\alpha}$ are all isotopic, and thus the boundary curves of any two regular neighbourhoods are isotopic.\\
\begin{Lema}\label{HTwelldefined}
 The map $\widehat{\psi}$ is well-defined.
\end{Lema}
\begin{proof}
Let $\alpha$ be a separating curve and $X_{1}$ and $X_{2}$ be two defining chains of $\alpha$. We divide this proof in two parts, depending on whether $X_{1}$ and $X_{2}$ are in the same connected component of $S_{1} \backslash \{\alpha\}$ or not.\\
\textit{Part 1:} If $X_{1}$ and $X_{2}$ are in two different connected components, then due to Corollary \ref{halfmultipants} we have that every element in $\widehat{\psi}(X_{1}) = \widetilde{\psi}(X_{1})$ will be disjoint from every element in $\widehat{\psi}(X_{2}) = \widetilde{\psi}(X_{2})$; now, if $X_{1}$ (and thus also $\widehat{\psi}(X_{1})$) has length $2k$, then $X_{2}$ (and thus also $\widehat{\psi}(X_{2})$) has length $2(g-k)$. If we cut $S_{1}$ along the boundary curve of the regular neighbourhood of $\widehat{\psi}(X_{1})$, we obtain a surface $S_{1}^{\prime}$ that has two connected components, one of genus $k$ and another of genus $g-k$. If we cut $S_{1}^{\prime}$ along the boundary curve of a regular neighbourhood of $\widehat{\psi}(X_{2})$ (which means we are cutting $S_{1}^{\prime}$ in the connected component of genus $g-k$), we obtain a surface with three connected components: one of genus $k$ (since it is where the elements of $\widehat{\psi}(X_{1})$ are contained), one of genus $g-k$ (since it is where the elements of $\widehat{\psi}(X_{2})$ are contained), and an annulus. Therefore the two boundary curves of the regular neighbourhoods are isotopic, i.e. $\widehat{\psi}(\alpha)$ is well defined for these two chains.\\
\textit{Part 2:} If $X_{1}$ and $X_{2}$ are in the same connected component, then we can find a defining chain $X_{3}$ on the other connected component such that the pairs $(X_{1}, X_{3})$ and $(X_{2}, X_{3})$ satisfy the conditions of the previous part, so the boundary curves of the regular neighbourhoods of the chains $(\widehat{\psi}(X_{1}), \widehat{\psi}(X_{3}))$ and $(\widehat{\psi}(X_{2}),\widehat{\psi}(X_{3}))$ are isotopic. Therefore $\widehat{\psi}(\alpha)$ is well defined.
\end{proof}
\indent Now we prove that $\widehat{\psi}$ is an edge-preserving map, so that we can apply Theorem A from \cite{JHH2}.
\begin{Lema}\label{hatpsiedge}
 $\widehat{\psi}$ is an edge-preserving map.
\end{Lema}
\begin{proof}
 \indent What we must prove is that given $\alpha$ and $\beta$ two disjoint curves, then $\widehat{\psi}(\alpha)$ and $\widehat{\psi}(\beta)$ are disjoint. If both $\alpha$ and $\beta$ are nonseparating curves, then we get the result from Corollary \ref{halfmultipants}. If $\alpha$ is nonseparating and $\beta$ is separating, let $X$ be a defining chain of $\beta$ such that $\alpha \in X$. Then by definition $\widehat{\psi}(\alpha)$ is disjoint from $\widehat{\psi}(\beta)$.\\
 \indent If $\alpha$ and $\beta$ are both separating, then we can always find two disjoint defining chains $X_{\alpha}$ and $X_{\beta}$ of $\alpha$ and $\beta$ respectively. Then by the two previous cases, every element of $\widehat{\psi}(X_{\alpha})$ is disjoint from every element of $\widehat{\psi}(X_{\beta}) \cup \{\widehat{\psi}(\beta)\}$ and every element of $\widehat{\psi}(X_{\beta})$ is disjoint from every element of $\widehat{\psi}(X_{\alpha}) \cup \{\widehat{\psi}(\alpha)\}$. Since by definition $\widehat{\psi}(\alpha)$ is the boundary curve of a regular neighbourhood of $\widehat{\psi}(X_{\alpha})$, if $\widehat{\psi}(\beta)$ and $\widehat{\psi}(\alpha)$ were to intersect each other, $\widehat{\psi}(\beta)$ would have to intersect at least one element of $\widehat{\psi}(X_{\alpha})$; thus $i(\widehat{\psi}(\beta), \widehat{\psi}(\alpha)) = 0$.\\
 \indent To prove that $\widehat{\psi}(\alpha) \neq \widehat{\psi}(\beta)$, let $X_{1}$ and $X_{2}$ be chains such that $X_{1}$ is a defining chain of $\alpha$ and $X_{1} \cup X_{2}$ is a defining chain of $\beta$. Thus $\widehat{\psi}(X_{1})$ and $\widehat{\psi}(X_{1} \cup X_{2})$ are defining chains of $\widehat{\psi}(\alpha)$ and $\widehat{\psi}(\beta)$ respectively. This implies that there exists (by the first case) an element in $\widehat{\psi}(X_{1} \cup X_{2})$ that is disjoint from every element in $\widehat{\psi}(X_{1})$, and another element in $\widehat{\psi}(X_{1} \cup X_{2})$ that intersects at least one element in $\widehat{\psi}(X_{1})$ (this happens since $\widehat{\psi}|_{\gcomp{S_{1}}} = \widetilde{\psi}$ and we can apply Corollary \ref{halfmultipants}). Then by Lemma \ref{differentchains}, $\widehat{\psi}(\alpha) \neq \widehat{\psi}(\beta)$. Therefore they are disjoint.
\end{proof}
Now, for the sake of completeness, we first cite Theorem A from \cite{JHH2} and then finalize with the proof of Theorem \ref{TheoB}:
\begin{Teono}[A in \cite{JHH2}]
 Let $S_{1} = S_{g_{1},n_{1}}$ and $S_{2} = S_{g_{2},n_{2}}$ be two orientable surfaces of finite topological type such that $g_{1} \geq 3$, and $\kappa(S_{2}) \leq \kappa(S_{1})$; let also $\fun{\varphi}{\ccomp{S_{1}}}{\ccomp{S_{2}}}$ be an edge-preserving map. Then, $S_{1}$ is homeomorphic to $S_{2}$ and $\varphi$ is induced by a homeomorphism $S_{1} \rightarrow S_{2}$.
\end{Teono}
\textbf{Proof of Theorem \ref{TheoB}:} We apply Theorem A from \cite{JHH2} to $\widehat{\psi}$, obtaining an element $h$ of $\EMod{S_{1}}$ that induces it. Since $h|_{\gcomp{S_{1}}} = \widehat{\psi}|_{\gcomp{S_{1}}} = \widetilde{\psi}$, we have that (by Corollary \ref{cordefpsitilde}) for every cut system $C = \{\alpha_{1}, \ldots, \alpha_{g}\}$ in $S_{1}$, $\widetilde{\phi}(C) = \{h(\alpha_{1}), \ldots, h(\alpha_{g})\}$. Therefore $h$ induces $\widetilde{\phi}$.
\section{Proof of Corollary \ref{CoroC}}\label{chap4sec4}
\indent To prove Corollary \ref{CoroC}, we first prove a consequence of Theorem \ref{TheoB}:
\begin{Cor}\label{CorS2isS1}
 Let $S = S_{g,0}$ be an orientable closed surface of finite topological type of genus $g \geq 3$, and $\fun{\phi}{\htcomp{S}}{\htcomp{S}}$ be an \epa map. Then $\phi$ is induced by a homeomorphism of $S$.
\end{Cor}
\begin{proof}
 By supposing $S = S_{1} = S_{2}$ we have that $\piHT$ is the identity, and by applying Theorem \ref{TheoB} we obtain that $\widetilde{\phi} = \phi$ is induced by a homeomorphism.
\end{proof}
\begin{proof}[\textbf{Proof of Corollary \ref{CoroC}}]
 \indent Let $\fun{\phi}{\htcomp{S_{1}}}{\htcomp{S_{2}}}$ be an isomorphism. We first prove that it is alternating. By Lemma \ref{colourepa} we have that for all cut systems $C$ in $S_{1}$, $\phi$ preserves the colours in $\lk{C}$. Applying the same lemma to $\phi^{-1}$, we have that two cut systems in $\lk{C}$ are in the same colour if and only if their images are in the same colour in $\lk{\phi(C)}$. This implies that $\phi$ is an alternating map.\\
 \indent Given that $\phi$ and $\phi^{-1}$ are isomorphisms, then they are also edge-preserving maps, and by Theorem \ref{TheoA} applied to $\phi$ and $\phi^{-1}$, we have that $g_{1} = g_{2}$.\\
 \indent Now, let $S = S_{g,0}$ with $g \geq 3$. To prove that $\Aut{\htcomp{S}}$ is isomorphic to $\EMod{S}$, we note there is a natural homomorphism:
 \begin{center}
  \begin{tabular}{cccccc}
   $\Psi_{\htcomp{S}}:$ & $\EMod{S}$ & $\longrightarrow$ & $\mathrm{Aut}(\htcomp{S})$ & & \\
    & $[h]$ & $\mapsto$ & $\varphi:\htcomp{S}$ & $\rightarrow$ & $\htcomp{S}$\\
    & & & $\{\alpha_{1}, \ldots, \alpha_{g}\}$ & $\mapsto$ & $\{h(\alpha_{1}), \ldots, h(\alpha_{g})\}$
  \end{tabular}
 \end{center}
 \textit{Injectivity}: If $h_{1}$ and $h_{2}$ are two homeomorphisms of $S$ such that $\Psi_{\htcomp{S}}([h_{1}]) = \Psi_{\htcomp{S}}([h_{2}])$, then the action of $h_{1}$ and $h_{2}$ on the nonseparating curves on $S$ would be exactly the same. Recalling that Schmutz-Schaller proved in \cite{Schmutz} that $\mathrm{Aut}(\gcomp{S})$ is isomorphic to $\EMod{S}$, this implies that $h_{1}$ is isotopic to $h_{2}$.\\
 \textit{Surjectivity}: If $\phi \in \mathrm{Aut}(\htcomp{S})$, then (as was proved above) it is an edge-preserving alternating map. Thus, by Corollary \ref{CorS2isS1} we have that $\phi$ is induced by a homeomorphism.\\
 \indent Therefore, $\mathrm{Aut}(\htcomp{S})$ is isomorphic to $\EMod{S}$.
\end{proof}

\end{document}